\documentclass{amsart}

\usepackage{cgeometry}

\bibliography{AG}

\newcommand{\mm}{\mathbf{M}}
\newcommand{\dd}{\mathbf{D}}
\newcommand{\gga}{\mathbf{G}}
\DeclareMathOperator{\Dom}{Dom}

\begin{document}

\title[Sharp lower bounds for Calabi type functionals]{On sharp lower bounds for Calabi type functionals and destabilizing properties of gradient flows}

\author{Mingchen Xia}

\begin{abstract} Let $X$ be a compact K\"ahler manifold with a given ample line bundle $L$.
Donaldson proved one inequality between the Calabi energy of a K\"ahler metric in $c_1(L)$ 
and the negative of normalized Donaldson--Futaki invariants of test configurations of $(X,L)$. 
He also conjectured that the bound is sharp.

In this paper, we prove a metric analogue of Donaldson's conjecture, we show that if we enlarge the space of test configurations to the space of geodesic rays in $\mathcal{E}^2$ and replace the Donaldson--Futaki invariant by the radial Mabuchi K-energy $\mathbf{M}$, then a similar bound holds and the bound is indeed sharp.
Moreover, we construct explicitly a minimizer of $\mathbf{M}$.
On a Fano manifold, a similar sharp bound for the Ricci--Calabi energy is also derived.
\end{abstract}

\maketitle

\setcounter{tocdepth}{1}
\tableofcontents

\section{Introduction}
\subsection*{Motivation}

Let $(X,L)$ be a polarized manifold of dimension $n$, namely, $X$ is a compact complex manifold of dimension $n$ and $L$ is an ample line bundle on $X$. We fix a K\"ahler metric on $X$ in the class $c_1(L)$.
Let $\mathcal{H}$ be the space of smooth strictly $\omega$-psh functions on $X$.
It is well-known that $\mathcal{H}$ is a Fr\'echet--Riemann manifold of constant non-positive curvature with respect to the standard Mabuchi--Donaldson--Semmes $L^2$ metric structure. See \cite{Blo12} for details.

Donaldson (\cite{Don05}) proved the following inequality:
\begin{equation}\label{eq:Don1}
\inf_{\varphi\in \mathcal{H}} Ca(\varphi)\geq \max\left(\sup_{(\mathcal{X},\mathcal{L})}\frac{-\mathrm{DF}(\mathcal{X},\mathcal{L})}{\|(\mathcal{X},\mathcal{L})\|_{L^2}},0\right)\,,
\end{equation}
where $Ca$ is the Calabi functional, $(\mathcal{X},\mathcal{L})$ takes value in the set of non-trivial normal test configurations of $(X,L)$ with reduced central fibre, $\mathrm{DF}$ is the Donaldson--Futaki invariant of a test configuration. For the definition of the $L^2$ norm of a test configuration, see \cite{His16}.
Donaldson conjectured in the same paper that equality should hold.

To appreciate \eqref{eq:Don1}, we recall that $Ca(\varphi)=0$ if{f} $\varphi$ is a cscK metric, on the other hand the right-hand side  of \eqref{eq:Don1} is zero if{f} $(X,L)$ is K-semistable. So \eqref{eq:Don1} establishes a connection between the canonical metrics and the GIT stability.

In terms of non-Archimedean metrics introduced by Boucksom, Hisamoto, Jonsson (\cite{BHJ19}, \cite{BHJ17}), \eqref{eq:Don1} can be reformulated as (see Section~\ref{subsec:TC})
\begin{equation}\label{eq:Don}
\inf_{\varphi\in \mathcal{H}} Ca(\varphi)\geq \max\left(0, \sup_{\psi\in \mathcal{H}^{\NA}\setminus \{\phi_{\mathrm{triv}}\}}\frac{-M^{\NA}(\psi)}{\|\psi\|_{L^2}}\right)\,,
\end{equation}
where $\mathcal{H}^{\NA}$ is the space of non-Archimedean FS metrics on $(X,L)$ (i.e. a FS metric on the Berkovich analytification of $(X,L)$ with respect to the trivial norm on $\mathbb{C}$), $\phi_{\mathrm{triv}}$ denotes the trivial metric, $M$ is the Mabuchi K-energy, the super-index $\NA$ denotes the non-Archimedean version of a functional.

In the present paper, we will prove a metric analogue of Donaldson's conjecture. That is, we prove that equality holds in \eqref{eq:Don} if we enlarge $\mathcal{H}$ to $\mathcal{E}^2$ and $\mathcal{H}^{\NA}$ to $\mathcal{R}^2$ (the space of $\mathcal{E}^2$ geodesic rays) and if we replace the non-Archimedean functional $M^{\NA}$ by the corresponding radial functional $\mm$. 
We also prove an analogous result for the radial Ding functional $\dd$ and the Ricci--Calabi energy $R$. See Section~\ref{sec:inv} for the definitions of various functionals.

Recall that the space $\mathcal{E}^2$ is the metric completion of $\mathcal{H}$ with respect to the $L^2$ metric. It is a deep theorem of Darvas (previously conjectured by Guedj) that the space $\mathcal{E}^2$ can be concretely realized as a subset of $\PSH(X,\omega)$ consisting of $\omega$-psh functions with finite energy. See \cite{Gue14} for a survey of these facts.

\subsection*{Statement of the main result}

Our proof of the main result will rely on the gradient flows of $M$ and $D$, which we recall now. The definitions of various functionals will be recalled in Section~\ref{sec:inv}.

The gradient flow of $M$ is known as the Calabi flow:
\begin{equation}\label{eq:CF}
    \left\{
    \begin{aligned}
        \partial_t\varphi_t&=S(\varphi_t)-\bar{S}\,,\\
        \varphi_t|_{t=0}&=\varphi_0\,,
    \end{aligned}
    \right.
\end{equation}
where $S$ denotes the scalar curvature of a metric, $\varphi_0\in \mathcal{H}$ and
\[
\bar{S}=\frac{1}{V}\int_X S(\varphi)\omega^n_{\varphi}
\]
is independent of the choice of $\varphi\in \mathcal{H}$.

The main difficulty is that the equation is of 4-th order.
The short time existence of the solution is proved in \cite{CH08} using a general method of 4-th order quasi-linear parabolic equations. However, the long time existence is still widely open. Chen, Cheng (\cite{CC18b}) proved the existence of long-time solution under the assumption of the existence of \emph{a priori} bounds of the scalar curvature.

In contrast, if we enlarge the space $\mathcal{H}$ to the finite energy space $\mathcal{E}^2$, it is shown in \cite{BDL17} that the long time solution does exist and coincides with the smooth solution on the time interval where the latter exists. We refer to such a flow as the \emph{weak Calabi flow}. The study of the weak Calabi flow dates back to \cite{Str14} and \cite{Str16}.

In the Fano setting, namely, when $X$ is a Fano manifold and $L=-K_X$, the gradient flow of $D$ is known as the inverse Monge--Ampère flow:
\begin{equation}
    \left\{
    \begin{aligned}
        \partial_t\varphi_t&=1-e^{\rho_{t}}\,,\\
        \varphi_t|_{t=0}&=\varphi_0\,,
    \end{aligned}
    \right.
\end{equation}
where $\varphi_0\in \mathcal{H}$, $\rho$ denotes the Ricci potential, $\rho_t=\rho_{\varphi_t}$. See Section~\ref{sec:inv} for the precise definition.

The study of this flow is initiated recently by Collins, Hisamoto and Takahashi (\cite{CHT17}). 
A crucial advantage of this flow is that the flow equation is a second order parabolic equation, hence the short-time existence follows from the general theory. For the long time behaviour, the standard theory of Monge--Ampère equations reduces the long time existence to derive \emph{a priori} $C^0$ bound of $\varphi_t$. This is done by a compactness argument in \cite{CHT17}.

A key feature of the (weak) Calabi flow is that $M$ is convex along the flow. Hence, $Ca$ is decreasing along the flow and it makes sense to consider the limit value of $Ca$ along the flow. It is easy to prove that the limit value of $Ca$ does not depend on the initial value (see Proposition~\ref{prop:Bind}).

These remarks apply equally to the inverse Monge--Ampère flow with $D$ in place of $M$.

The main result of this paper is the following metric analogue of Donaldson's conjecture \eqref{eq:Don}. 

\begin{theorem}\label{thm:main}
Let $X$ be a compact K\"ahler manifold. Let $\omega$ be a K\"ahler form on $X$. Let $\mathcal{E}^2=\mathcal{E}^2(X,\omega)$, $\mathcal{H}=\mathcal{H}(X,\omega)$.

1. We have
\[
\inf_{\phi\in \mathcal{E}^2}Ca(\phi)=\max_{\ell\in \mathcal{R}^{2}\setminus\{0\}} \frac{-\mm(\ell)}{\|\ell\|}\,.
\]

2. In the Fano case,
\[
\inf_{\varphi\in \mathcal{H}}R(\varphi)=\max_{\ell\in \mathcal{R}^{2}\setminus\{0\}} \frac{-\dd(\ell)}{\|\ell\|}\,.
\]

Moreover, the inf in 1. (resp 2.) can be obtained as follows: let $\phi_0\in \mathcal{E}^2$ with $M(\phi_0)<\infty$ (resp. $\varphi_0\in \mathcal{H}$), let $\phi_t$ (resp. $\varphi_t$) be the weak Calabi flow (resp. inverse Monge--Amp\`ere flow) with initial value $\phi_0$ (resp. $\varphi_0$), then
\[
\inf_{\phi\in \mathcal{E}^2}Ca(\phi)=\lim_{t\to\infty}Ca(\phi_t)\,,\quad \inf_{\varphi\in \mathcal{H}}R(\varphi)=\lim_{t\to\infty}R(\varphi_t)\,.
\]
\end{theorem}
Notice that in our theorem, we do not require that the polarization of $X$ be integral anymore.

Here $\mathcal{R}^{2}$ is the space of geodesic rays in $\mathcal{E}^2$ emanating from a point $\varphi\in \mathcal{H}$. The norm $\|\ell\|$ of $\ell\in \mathcal{R}^2$ is defined as the $d_2$ distance between $\ell_0$ and $\ell_1$.
The notation $0$ is used for the constant geodesic.
According to the recent work of Darvas--Lu (\cite{DL18}), the max terms of both statements do not depend on the choice of $\varphi$. In the general context of Hadamard spaces, $\mathcal{R}^2$ is also known as the cone at infinity of $\mathcal{H}$ (\cite{Bal12}). For the definition of $Ca$ on $\mathcal{E}^2$, see Section~\ref{subsec:wcf}.
We also notice that by considering the following geodesic ray $(\varphi+t)_t\in \mathcal{R}^2$,  both max terms in Theorem~\ref{thm:main} are non-negative.

An abstract version of this result, which applies to general gradient flows in Hadamard spaces is also included, see Theorem~\ref{thm:abst}.

In Section~\ref{subsec:TC}, we explain the relation between Donaldson's conjecture and Theorem~\ref{thm:main}. 

Our proof is constructive. We construct a geodesic ray (called the \emph{Darvas--He geodesic ray}) following the method in \cite{DH17}, which was designed originally for the K\"ahler--Ricci flow. We calculate the radial $M$ or $D$ functional along this ray and show that this ray is indeed a maximizer.

In the unstable case, the situation is rather simple. We prove
\begin{corollary}\label{cor:main}
1. Assume that $(X,\omega)$ is geodesically unstable (Definition~\ref{def:geoduns}),
then there is a unique maximizer of $-\mm$ on the unit sphere in $\mathcal{R}^2$.

2. In the Fano case, assume that $X$ is K-unstable,
then there is a unique maximizer of $-\dd$ on the unit sphere in $\mathcal{R}^2$. 
\end{corollary}
\subsection*{Relations to other results}
In the toric setting, various special cases are already known. 

Part 2 of Theorem~\ref{thm:main} is proved in the toric setting in \cite{CHT17} Theorem~1.4, see also \cite{Yao17}.

As for Part 1 of Theorem~\ref{thm:main}, in the toric setting, it is proved in \cite{Sze08} (1). Moreover, assuming the long time existence of smooth solutions to the Calabi flow, the original version of Donaldson's conjecture is also proved in the toric setting in the same paper.

A similar result for the $H$ functional on Fano manifolds is proved in \cite{DS16}.

After finishing this paper, the author was informed that T. Hisamoto (\cite{His19}) has independently proved the Fano case of the main theorem. Moreover, in the Fano case, Hisamoto also proved that the max in Theorem~\ref{thm:main} can be obtained by a sequence of test configurations.

After the first version of this paper on arXiv, there have been a number of related papers about optimal distabilizing properties in various settings. See \cite{BLZ19}, \cite{Der19}, \cite{Tak19}, \cite{SD19}.

\subsection*{Acknowledgement}
The author benefited from discussions with Robert Berman, Tamás Darvas, Jiaxiang Wang, Tomoyuki Hisamoto and Miroslav Bačák. The author would like to thank S\'ebastien Boucksom for pointing out a mistake in the arXiv version and the anonymous referee for suggestions to improve the presentation of the paper.

\section{Preliminaries on K\"ahler geometry, pluripotential theory and Mabuchi geometry}\label{sec:inv}
Let $X$ be a compact polarized manifold of dimension $n$. Let $\omega$ be a K\"ahler form on $X$.
We will frequently consider the special case where $X$ is Fano and $\omega\in c_1(X)$, which we refer to as the Fano case.

Set $V=\int_X \omega^n$.
Let $\mathcal{H}$ be the space of smooth strictly $\omega$-psh functions with the usual Mabuchi--Semmes--Donaldson $L^2$-metric: take $f,g\in C^{\infty}(X)=T_{\varphi}\mathcal{H}$ for some $\varphi\in \mathcal{H}$, define
\[
\langle f,g\rangle_{\varphi}=\frac{1}{V}\int_X fg\,\omega_{\varphi}^n\,.
\]
It is well-known that $\mathcal{H}$ is a Fr\'echet--Riemann manifold of constant non-positive curvature. See \cite{Blo12} for details.

Given $\varphi\in \mathcal{H}$, write $\omega_{\varphi}=\omega+\ddc \varphi$, where we use the convention 
\[
\ddc:=\frac{\mathrm{i}}{2\pi}\partial \overline{\partial}\,.
\]

\subsection{Finite energy class}
It is proved by Darvas (\cite{Dar15}) that the metric completion of $\mathcal{H}$ with respect to the $L^2$ metric can be realized by the set $\mathcal{E}^2$ of finite energy $\omega$-psh functions. We briefly recall the related definitions. 

We define
\[
\mathcal{E}(X,\omega)=\left\{\,\varphi\in \PSH(X,\omega): \int_X  \omega_{\varphi}^n =V \,\right\}\,.
\]
Here and in the sequel, the product $\omega_{\varphi}^n$ is always interpreted in the non-pluripolar sense of \cite{BEGZ10}.

Define the following classes for $1\leq p<\infty$
\[
\mathcal{E}^p:=\left\{ \,\varphi\in \mathcal{E}(X,\omega): \int_{X}|\varphi|^p\, \omega_{\varphi}^n<\infty\, \right\}\,.
\]
We also define $\mathcal{E}^{\infty}$ to be the set of bounded $\omega$-psh functions on $X$.

According to Chen (\cite{Chen00}), for any $\varphi_0,\varphi_1\in \mathcal{H}$, there is a unique weak geodesic connecting $\varphi_t$ connecting them. According to a recent regularity result (\cite{CTW17}), this weak geodesic has $C^{1,1}$-regularity. One can define a distance $d_p$ on $\mathcal{H}$ for each $p\in [1,\infty)$ by
\begin{equation}
d_p(\varphi_0,\varphi_1)=\left(\frac{1}{V}\int_X |\dot{\varphi}_0|^p\,\omega_{\varphi_0}^n\right)^{1/p}\,.
\end{equation}
It is shown in \cite{Dar15} Theorem~3.5 that $d_p$ is indeed a metric on $\mathcal{H}$. However, this metric is not complete. It is natural to look for the metric completion of $d_p$. In the same paper \cite{Dar15}, Darvas proved that the metric completion of $\mathcal{H}$ with respect to $d_p$ can be realized as $\mathcal{E}^p$. 
For the definition of $d_p$ on $\mathcal{E}^p$, we refer to \cite{Dar15} (5).
Moreover, $\mathcal{E}^p$ is indeed a geodesic metric space (\cite{Dar15} Theorem~4.17). We will recall some related definitions below in Section~\ref{subsec:wg} and Section~\ref{subsec:mg}.

Recall for $\varphi,\psi\in \mathcal{E}^2$, we have
\begin{equation}\label{eq:ulb1}
C^{-1}I_p(\varphi,\psi)\leq d_p(\varphi,\psi)\leq C I_p(\varphi,\psi)\,,
\end{equation}
where $C>0$ is a universal constant and
\[
I_p(\varphi,\psi)=\left(\int_X |\varphi-\psi|^p\,\omega_{\varphi}^n \right)^{1/p}+\left(\int_X |\varphi-\psi|^p \,\omega_{\varphi}^n \right)^{1/p}\,.
\]
For a proof, see \cite{Dar15} Theorem~3.

The metric topology on $\mathcal{E}^1$ is also known as the strong topology. It is studied in detail in \cite{BBEGZ16}. In this case, the topology admits a very explicit description.

Recall that the usual Monge--Ampère energy $E:\mathcal{H}\rightarrow \mathbb{R}$ (See~\eqref{eq:E}) extends to $E:\mathcal{E}^1\rightarrow \mathbb{R}$.
The functional is concave, increasing. See \cite{BB10} Section~3 for example.
The strong topology on $\mathcal{E}^1$ is then the coarsest refinement of the $L^1$-topology that makes $E$ continuous. For the proof of this fact, see \cite{Dar15} Proposition~5.9.

We refer to \cite{Dar19} for a systematic introduction to this material.
\subsection{Functionals}\label{subsec:fun}
Let $E:\mathcal{H}\rightarrow \mathbb{R}$ be the Monge--Ampère energy functional:
\begin{equation}\label{eq:E}
E(\varphi)=\frac{1}{(n+1)V}\sum_{j=0}^n \int_X \varphi \, \omega^j \wedge \omega_\varphi^{n-j}\,.
\end{equation}
This functional extends to a concave, increasing functional on $\mathcal{E}^1$ in a natural way. See \cite{BB10} Section~3.

Define the Calabi energy $Ca:\mathcal{H}\rightarrow \mathbb{R}$ as
\begin{equation}\label{eq:defCa}
Ca(\varphi)=\left(\frac{1}{V}\int_X (S(\varphi)-\bar{S})^2\,\omega_{\varphi}^n\right)^{1/2}\,,
\end{equation}
where $S(\varphi)$ is the scalar curvature of $\varphi$ and \[
\bar{S}=\frac{1}{V}\int_X S_{\varphi}\,\omega_{\varphi}^n
\]
is independent of the choice of $\varphi\in \mathcal{H}$. Note that in most literature, Calabi energy is defined as $(Ca)^2$.

We will show in Section~\ref{subsec:wcf} that $Ca$ has a natural lsc extension to $\mathcal{E}^2\rightarrow (-\infty,\infty]$.

Recall the definition of $E_R:\mathcal{H}\rightarrow \mathbb{R}$:
\begin{equation}
E_R(\varphi)=\frac{1}{nV}\sum_{j=0}^{n-1}\int_X \varphi\Ric \omega\wedge \omega_{\varphi}^j\wedge \omega^{n-1-j}\,.
\end{equation}
As in \cite{BDL17} Section~4.2, this functional extends naturally to a continuous functional $E_R:\mathcal{E}^1\rightarrow \mathbb{R}$.

Recall the definition of the entropy $H:\mathcal{H}\rightarrow \mathbb{R}$:
\begin{equation}
H(\varphi)=\frac{1}{V}\int_X \log \frac{\omega_{\varphi}^n}{\omega^n} \,\omega^n_{\varphi}\,.
\end{equation}
This functional extends naturally to $H:\mathcal{E}^1\rightarrow [0,\infty]$.

Let us also recall the definition of the Mabuchi functional $M:\mathcal{H}\rightarrow \mathbb{R}$:
\begin{equation}
M(\varphi)=H(\varphi)+\bar{S}E(\varphi)-nE_R(\varphi)\,.
\end{equation}
We have extended every term, hence we get $M:\mathcal{E}^1\rightarrow (-\infty,\infty]$. The extension is lsc and convex along finite energy geodesics. See \cite{BDL17} Theorem~4.7, \cite{BB17}, \cite{CLP14}  for details.

In the Fano setting, we have two more functionals $R$ and $D$.

Let $D:\mathcal{H}\rightarrow \mathbb{R}$ be the Ding functional. 
Recall that by definition, this means
\begin{equation}\label{eq:defD}
\left\{
\begin{aligned}
\delta D(\varphi)&=\frac{1}{V}(e^{\rho_{\varphi}}-1)\omega_{\varphi}^n\,,\\
D(\omega)&=0\,.
\end{aligned}
\right.
\end{equation}
where $\rho_{\varphi}$ is the Ricci potential of $\varphi$:
\begin{equation}\label{eq:defrho}
\left\{
\begin{aligned}
\Ric \omega_{\varphi}-\omega_{\varphi} &= \ddc\rho_{\varphi}\,,\\
\int_{X} \left(e^{\rho_\varphi}-1\right)\omega_{\varphi}^n &=0\,.
\end{aligned}
\right.
\end{equation}

More explicitly, this means
\begin{equation}
D(\varphi)=-E(\varphi)-\log \int_X e^{-\varphi+\rho}\omega^n\,,
\end{equation}
where $\rho$ is the Ricci potential of $\omega$.

This formula then extends directly to $\mathcal{E}^1\rightarrow \mathbb{R}$. The extension is continuous and convex along finite energy geodesics. We refer to \cite{Bernd09}, \cite{Bernd15}, \cite{Dar17} Chapter~4 for details.

Define the Ricci--Calabi energy $R:\mathcal{H}\rightarrow \mathbb{R}$ as
\[
R(\varphi)=\left(\frac{1}{V}\int_X (e^{\rho_{\varphi}}-1)^2\omega_{\varphi}^n\right)^{1/2}\,.
\]
\subsection{The space of weak geodesic rays}\label{subsec:wg}
In this section, we recall some notions from the very recent work of Darvas--Lu (\cite{DL18}).

We first recall the definition of (weak) geodesics. 

Let $\Delta(r)\subset \mathbb{C}$ be the open disc of radius $r$ centered at $0$. Let $\Delta=\Delta(1)$. Let $\Delta^*=\Delta\setminus\{0\}$. Let $\pi:X\times \Delta^*\rightarrow X$ be the natural projection. 

Let $\ell_t$ ($t\in [0,a]$, $a\in(0,\infty]$) be a ray or segment in $\mathcal{E}^{\infty}(X,\omega)$.
Define $D=\bar{\Delta}\setminus\Delta(e^{-a})$.
The complexification $\Phi$ of $\ell_t$ is by definition a function on $X\times D$, such that $\Phi_s=\ell_{-\log |s|}\,,s\in D$.
When $\Phi$ is $\pi^*\omega$-psh and solves the homogeneous Monge--Ampère equation
\[
(\pi^*\omega+\ddc \Phi)^{n+1}=0 \quad \text{on }X\times \mathrm{Int}\,D \,,
\]
we call $\ell$ a weak geodesic.
Similarly, $\ell_t$ is called a subgeodesic, if $\Phi$ is just $\pi^*\omega$-psh.

For two points $\varphi,\psi\in \mathcal{H}$, there is a unique (up to normalization) weak geodesic segment connecting $\varphi$ and $\psi$, the geodesic segment has $C^{1,1}$ regularity (\cite{CTW17}).

In general, for any two points $\varphi,\psi\in \mathcal{E}^p$ ($p\in [1,\infty]$), we may take a Demailly approximation, namely, decreasing sequences $\varphi_j$, $\psi_j$ in $\mathcal{H}$, converging to $\varphi$ and $\psi$ respectively. Then the geodesic segment connecting $\varphi_j$ and $\psi_j$ converge to a unique segment in $\mathcal{E}^p$, which does not depend on the choice of $\varphi_j$ and $\psi_j$. The limit is known as the finite energy geodesic segment in $\mathcal{E}^p$ connecting $\varphi$ and $\psi$. 
The finite energy geodesic is indeed a $d_p$-metric geodesic. Moreover, $\mathcal{E}^p$ is a geodesic metric space. The definitions of a metric geodesic and a geodesic metric space are recalled in Section~\ref{subsec:mg}.
It is known that the $d_p$-metric geodesic between points in $\mathcal{E}^p$ when $p>1$ is unique, so in these cases (\cite{DL18}), we use the term \emph{geodesic} instead of \emph{finite energy geodesic}. Note however that, the $d_1$-geodesics are not unique in general.

Now a ray $\ell_t$ ($t\geq 0$) in $\mathcal{E}^p$ is called a \emph{finite energy geodesic ray} in $\mathcal{E}^p$ emanating from $\ell_0$ if for any $s_2>s_1\geq 0$, the restriction of $\ell$ to $[s_1,s_2]$ is a finite energy geodesic segment in $\mathcal{E}^p$.

Let $\varphi\in \mathcal{H}$. Let $\mathcal{R}^p_{\varphi}$ be the set of finite energy geodesic rays in $\mathcal{E}^p$ emanating from $\varphi$. There is a special ray, namely the constant geodesic. This ray will be referred to as \emph{the origin}. We sometimes use the notation $0$ for the origin.

Define the chordal metric on $\mathcal{R}^p_{\varphi}$ as follows: let $\ell^1$ and $\ell^2$ be two elements in $\mathcal{R}^p_{\varphi}$, the distance is defined by
\begin{equation}
d^c_p(\ell^1,\ell^2):=\lim_{t\to\infty}\frac{d_p(\ell^1_t,\ell^2_t)}{t}\,.
\end{equation}
Now assume that $1\leq p<\infty$, then
 $(\mathcal{R}^p_{\varphi},d^c_p)$ is a complete geodesic metric space (\cite{DL18} Theorem~4.7, Theorem~4.9).

For any $\varphi,\psi\in \mathcal{E}^p$, there is a canonical isometry
\[
P_{\varphi,\psi}:\mathcal{R}^p_{\varphi} \rightarrow \mathcal{R}^p_{\psi}
\]
mapping each finite energy geodesic ray $\ell$ emanating from $\varphi$ to the unique parallel finite energy geodesic ray $\ell'$ emanating from $\psi$ (\cite{DL18} Theorem~1.3). Here parallel means that $d_p(\ell_t,\ell'_t)$ is bounded.
Moreover, if $\ell^0\in \mathcal{R}^p_{\varphi}$ and $\ell^1 \in \mathcal{R}^p_{\psi}$ are parallel, the radial functional $\mm$ (resp. $\dd$) to be defined in  Section~\ref{subsec:rad} takes same value on $\ell^1$ and $\ell^2$ if $M(\varphi),M(\psi)<\infty$ (resp. no restriction for $D$). See \cite{DL18} Lemma~4.10.

Hence, for our purpose, we simply identify $\mathcal{R}^p_{\varphi}$ for various $\varphi$ and write $\mathcal{R}^p$ when $p<\infty$.

Now $\mathcal{R}^p_{\varphi}$ forms a decreasing chain indexed by $p$. We know that $\mathcal{R}^{\infty}_{\varphi}$ is dense in arbitrary $\mathcal{R}^p_{\varphi}$ (\cite{DL18} Theorem~1.5).

\subsection{Radial functionals}\label{subsec:rad}
As $M$ and $D$ are both convex along finite energy geodesics, it is natural to define the radial version of these functionals. Fix $\varphi\in \mathcal{E}^1$.

Define $\mm:\mathcal{R}^1_{\varphi} \rightarrow (-\infty,\infty]$ by
\begin{equation}
\mm(\ell):=\lim_{t\to\infty}\frac{M(\ell_t)}{t}\,.
\end{equation}
Similarly, in the Fano case, define
$\dd:\mathcal{R}^1_{\varphi} \rightarrow (-\infty,\infty]$ by
\begin{equation}
\dd(\ell):=\lim_{t\to\infty}\frac{D(\ell_t)}{t}\,.
\end{equation}

We also define the $p$-energy of $\ell\in \mathcal{R}^p$ as follows:
\begin{equation}\label{eq:pnorm}
\|\ell\|_p:=E_p(\ell):=d^c_p(\ell,0)\,.
\end{equation}
Here $0$ denotes the constant geodesic. When $p=2$, we omit the subindex $2$.

Let $\ell_t$ ($t\in [0,s]$, $s>0$) be a weak geodesic segment between $\ell_0,\ell_s\in \mathcal{H}$. We define
\begin{equation}
\|\ell\|=E_2(\ell):= \left(\frac{1}{V}\int_X |\dot{\ell}_t|^2 \omega_{\ell_t}^n\right)^{1/2}
\end{equation}
for any $t\in [0,s]$. It is well-known that this definition does not depend on the choice of $t$ and is equal to $s^{-1}d_2(\ell_0,\ell_s)$.
See \cite{Dar15} Lemma~4.11.

\section{Preliminaries on metric geometry and gradient flows}\label{sec:weakg}
In this section, we review some basic facts about weak gradient flows on Hadamard spaces. We refer to \cite{Bac14}, \cite{AGS08}, \cite{Bac18} for details.

\subsection{Metric geometry}\label{subsec:mg}
We review several basic definitions from metric geometry.

Let $(M,d)$ be a metric space. A \emph{path} in $M$ is an element in $C^0([0 ,1],M)$. Let $\gamma$ be a path in $M$, the \emph{length} of $\gamma$ is defined as
\[
\ell(\gamma):=\sup \sum_{i=1}^n d(\gamma_{t_{i-1}},\gamma_{t_i})\,,
\]
where the sup is taken over the set of partitions $0=t_0<t_1<\cdots<t_n=1$ for various $n\in \mathbb{Z}_{>0}$.

The metric space $(M,d)$ is a \emph{length space} if for any $x,y\in M$, for any $\epsilon>0$, there is a path $\gamma$ in $M$ with $\gamma_0=x$, $\gamma_1=y$ and
\[
\ell(\gamma)\leq d(x,y)+\epsilon\,.
\]

A path $\gamma$ in $M$ is called a \emph{geodesic} if 
\[
d(\gamma_s,\gamma_t)=d(\gamma_0,\gamma_1)|s-t|
\]
for any $s,t\in [0,1]$.

The metric space $(M,d)$ is a \emph{geodesic space} if for any $x,y\in M$, there is a geodesic $\gamma$ with $\gamma_0=x$, $\gamma_1=y$.

From now on, we always assume that $(M,d)$ is a geodesic space.
A \emph{geodesic triangle} with vertices $x,y,z\in M$ consists of three geodesics $g_{xy}$, $g_{yz}$, $g_{zx}$, joining $x$ to $y$, $y$ to $z$, $z$ to $x$ respectively. The triangle will be denoted as $\Delta(x,y,z)$ although it is not uniquely determined by $x,y,z$. A \emph{companion triangle} $\Delta(\bar{x},\bar{y},\bar{z})$ of $\Delta(x,y,z)$ is a triangle in $\mathbb{R}^2$, whose vertices are denoted as $\bar{x},\bar{y},\bar{z}$, such that
\[
|\bar{x}-\bar{y}|=d(x,y)\,,\quad |\bar{y}-\bar{z}|=d(y,z)\,,\quad |\bar{z}-\bar{x}|=d(x,y)\,.
\]
Let $w$ be a point on the geodesic $g_{xy}$. The \emph{companion point} of $w$ is a point $\bar{w}$ on the line segment from $\bar{x}$ to $\bar{y}$, such that
\[
d(w,y)=|\bar{w}-\bar{y}|\,.
\]
Similarly one can define the companion point of a point on $g_{yz}$ and $g_{zx}$.

The geodesic metric space $(M,d)$ is a \emph{CAT(0) space} if for any geodesic triangle $\Delta(x,y,z)$ in $M$ with companion triangle $\Delta(\bar{x},\bar{y},\bar{z})$, for any $a$ on $g_{xy}$, $b$ on $g_{xz}$ with companion points $\bar{a}$, $\bar{b}$, we have
\[
d(a,b)\leq |\bar{a}-\bar{b}|\,.
\]

Geometrically, the CAT(0) condition means that $(M,d)$ has non-positive curvature. See \cite{Bac14} for a detailed explanation.

The geodesic metric space $(M,d)$ is a \emph{Hadamard space} if it is complete and is a CAT(0) space.

Examples of Hadamard spaces include complete Riemannian manifolds of non-positive curvature, the space $\mathcal{E}^2$, Hilbert spaces, e.t.c..

We recall the concept of weak convergence (also called $\Delta$-convergence) in a Hadamard space. See \cite{KP08} for a thorough treatment.
Let $(M,d)$ be a Hadamard space. Let $x_n\in M$ be a bounded sequence. For $x\in M$, define
\[
r(x):=\varlimsup_{n\to\infty} d(x,x_n)\,.
\]
The \emph{asymptotic radius} of $(x_n)$ is defined as $\inf_{x\in M}r(x)$. The \emph{asymptotic center} of $(x_n)$ is defined as the set
\[
\left\{\, x\in M:r(x)=\inf_{y\in M}r(y)\,\right\}\,.
\]
According to \cite{DKS06} Proposition~7, the set consists of a single element. By abuse of language, we also call this element the asymptotic center  of $(x_n)$. If $x\in M$ is the asymptotic center of every subsequence of $(x_n)$, we say that $(x_n)$ \emph{converges weakly} (or $\Delta$-\emph{converges}) to $x$. 

\begin{proposition}\label{prop:weakineq}
Let $(M,d)$ be a Hadamard space. Assume that $x_n\in M$ is a sequence that converges weakly to $x\in M$. Let $y\in M$, then
\begin{equation}\label{eq:noname}
d(y,x)\leq \varliminf_{n\to\infty} d(y,x_n)\,.
\end{equation}
\end{proposition}
This proposition is a special case of \cite{Bac13} Lemma~3.1, which says that a convex lsc function on a Hadamard space is weakly lsc.

\subsection{Weak gradient flows on Hadamard spaces}
In this subsection, following \cite{Bac14} Chapter 5, we explore the general theory of weak gradient flows on Hadamard spaces.

Let $(M,d)$ be a Hadamard space. Let $G:X\rightarrow (-\infty,\infty]$ be a convex lsc function. We will use the notation
\[
\Dom G=G^{-1}(\mathbb{R})\,.
\]
The \emph{slope} of $G$ is a function $|\partial G|:M\rightarrow [0,\infty]$:
\[
|\partial G|(y)=
\left\{
\begin{aligned}
\varlimsup_{z\to y}\frac{\max\{G(y)-G(z),0\}}{d(y,z)}\,,& \quad y\in \Dom(G)\,,\\
\infty\,,&\quad y\in G^{-1}(\infty)\,.
\end{aligned}
\right.
\]
It is a general fact that $|\partial G|$ is always lsc.  Moreover
\begin{equation}\label{eq:partialg}
|\partial G|(y)=\sup_{z\in M-\{y\}}\frac{\max\{G(y)-G(z),0\}}{d(y,z)}\,, \quad y\in \Dom(G)\,.
\end{equation}
See \cite{Bac14} Lemma~5.1.2 for a proof.

Inspired by the gradient flow on Hilbert spaces, we look for a gradient flow on a general Hadamard space as follows: given $c_0\in \Dom(G)$, we want to define a curve $c_t$ so that 
\[
|\dot{c}_t|:=\lim_{s\to t+} \frac{d(c_t,c_s)}{s-t}
\]
is as large as possible. That is, we hope that
\[
|\dot{c}_t|=|\partial G(c_t)|\,,\quad t>0\,.
\]
This is indeed possible, we recall the construction.

We define $c^{m,j}:[0,\infty)\rightarrow M$ ($m,j\in \mathbb{Z}_{\geq 0}$) by iteration:

1. $c^{m,0}_t=c_0$.

2. $c^{m,j+1}_t$ is the minimizer of
\[
v\mapsto \frac{1}{2}d(v,c_t^{m,j})^2+\frac{t}{m}G(v)\,.
\]

Set $c_t^m=c_t^{m,m}$. Set
\[
c_t=\lim_{m\to\infty}c_t^m\,.
\]

It is shown by Mayer (\cite{May98}) that the above procedure is well-defined, $c_t\in \Dom(G)$. The curve $c_t$ is called the \emph{weak gradient flow} of $G$ starting from $c_0$. See also \cite{Bac14} Theorem~5.1.6.

The curve $c_t$ has the following property:
\begin{equation}\label{eq:derg}
-\ddt G(c_t)=|\partial G(c_t)|^2=\left| \dot{c}_t \right|^2<\infty\,, \quad t> 0\,.
\end{equation}
Here the derivative on the left-hand side  is understood as the right derivative. In particular, $G(c_t)$ is right differentiable at $t>0$.
See \cite{Bac14} Theorem~5.1.13, \cite{AGS08} Theorem 2.4.15.
By \cite{Bac14} Proposition~5.1.14, $|\partial G(c_t)|$ is decreasing in $t\geq 0$, so $G(c_t)$ is convex in $t\geq 0$.

Moreover, the following evolution variation inequality holds (\cite{Bac14} Theorem~5.1.11)
\begin{equation}\label{eq:evv}
\frac{1}{2}\ddt d(c_t,v)^2\leq G(v)-G(c_t)\,,
\end{equation}
where $v\in \Dom(G)$. Here the left-hand side  is understood as the right upper derivative (Dini derivative), namely
\[
\ddt d(c_t,v)^2:=\varlimsup_{s\to t+}\frac{d(c_s,v)^2-d(c_t,v)^2}{s-t}\,.
\]
\begin{remark}
In \cite{Bac14}, this theorem is stated
for usual derivative and for almost all $t$. Moreover, it is shown that $d^2(c_t,v)$ is absolutely continuous. Our formulation follows easily from taking Dini derivative of the integral version of the theorem in \cite{Bac14}.
\end{remark}

Now fix a weak gradient flow $c_t$ with $c_0\in \Dom(G)$.
\begin{proposition}\label{prop:ulb}
Let $0< t<s$, then
\begin{equation}\label{eq:ulb}
|\partial G|(c_t) d(c_t,c_s)\geq G(c_t)-G(c_s)\geq |\partial G|(c_s)d(c_t,c_s)\,.
\end{equation}
Moreover, for $t=0$, the left-hand part of \eqref{eq:ulb} is still true, namely
\[
|\partial G|(c_0) d(c_0,c_s)\geq G(c_0)-G(c_s)\,.
\]
\end{proposition}
\begin{proof}
The left-hand part of \eqref{eq:ulb} (including the case $t=0$) follows directly from \eqref{eq:partialg}. 

We prove the right-hand part.
To prove \eqref{eq:ulb}, without loss of generality, assume that $t=0$, that \eqref{eq:derg} holds also at $t=0$ and that $c_t$ is Lipschitz on $[0,\infty)$ (\cite{Bac14} Proposition~5.1.10).

Define two functions
\[
F(r)=\left(G(c_0)-G(c_r)\right)^2,\, L(r)=d(c_0,c_r)^2  \,,\quad r\geq 0\,,
\]

We may assume that $L(s)>0$, since otherwise, by \cite{Bac14} Proposition~5.1.14, $|\partial G|(c_t)$ is constant for $t\in [0,s]$, hence by \eqref{eq:derg} and the fact that $c_0=c_s$, this constant is indeed $0$. So the flow $c_t$ is just the constant at $c_0$, the result is obvious.

Define a function $H:[0,s]\rightarrow \mathbb{R}$ as follows:
\[
H(a)=F(a)-\frac{F(s)}{L(s)}L(a)\,.
\]
Obviously, $H(0)=H(s)=0$, $H$ is a usc function. Let $x\in [0,s)$ be a maximizer of $H$. 
Then the right upper derivative of $H$ at $x$ must be non-positive, namely
\[
0\geq \varlimsup_{y\to x+} \frac{H(y)-H(x)}{y-x}=F'(x)-\frac{F(s)}{L(s)}\varliminf_{y\to x+} \frac{L(y)-L(x)}{y-x}\geq F'(x)-\frac{F(s)}{L(s)} L'(x)\,,
\]
where in the second step, we made use of the fact that $F$ is right differentiable since $G$ is also right differentiable, as recalled after \eqref{eq:derg}. Here each derivative denotes the right upper derivative.

Since $G$ is right differentiable, by \eqref{eq:derg}, we have
\[
F'(x)=2(G(c_0)-G(c_x))|\partial G(c_x)|^2\geq 0 \,.
\]
By \eqref{eq:evv}, we also have
\[
0\leq L'(s)\leq 2\left(G(c_0)-G(c_s)\right)\,. 
\]
When $L'(x)=0$, we conclude $F'(x)=0$ as well. Hence either $F(x)=0$ or $|\partial G(c_x)|=0$. In both cases, \eqref{eq:ulb} is obvious.
When $L'(x)>0$,
\[
\frac{F(s)}{L(s)}L'(x)\geq F'(x) \geq L'(x)|\partial G(c_x)|^2\,.
\]
This concludes the proof of \eqref{eq:ulb} since $|\partial G(c_x)|$ is decreasing in $x$ (\cite{Bac14} Proposition~5.1.14).
\end{proof}

\begin{proposition}\label{prop:Bind}
Let $\phi_0,\psi_0\in \Dom(G)$. Let $\phi_t$ (resp. $\psi_t$) be the weak gradient flow of $G$ with initial value $\phi_0$ (resp. $\psi_0$). Then
\[
\lim_{t\to\infty} |\partial G|(\phi_t)=\lim_{t\to\infty} |\partial G|(\psi_t)\,.
\]
\end{proposition}
This is proved in \cite{He15} Corollary~2.2.
\begin{proof}
We may assume that the curves $\phi_t$ and $\psi_t$ do not intersect. 
Moreover, we may assume that \eqref{eq:derg} holds up to $t=0$.
Assume that the conclusion is not true, we may assume that there is a constant $\delta>0$, so that for all $t\geq 0$
\[
|\partial G|^2(\phi_t)\leq |\partial G|^2(\psi_t)-\delta\,.
\]

Now by \eqref{eq:evv},
\[
2(G(\psi_t)-G(\phi_{t+1}))\geq d(\psi_t,\phi_{t+1})^2-d(\psi_t,\phi_t)^2\geq -d(\psi_0,\phi_0)^2\,,
\]
where we have used the fact that $d(\phi_t,\psi_t)\leq d(\phi_0,\psi_0)$ in the second inequality (\cite{Bac14} Theorem~5.1.6).

Now by \eqref{eq:derg},
\[
G(\phi_t)-G(\phi_{t+1})\leq |\partial G|^2(\phi_0)\,.
\]
By \eqref{eq:derg},
\[
\left(G(\psi_t)-G(\phi_t)\right)-\left(G(\psi_0)-G(\phi_0)\right)=\int_0^t \left(|\partial G(\phi_s)|^2- |\partial G(\psi_s)|^2\right)\,\mathrm{d}s
\leq -\delta t\,.
\]
In all, we get
\[
-d(\psi_0,\phi_0)\leq -2\delta t+C
\]
for some constant $C$. This is a contradiction.
\end{proof}

\subsection{Moment-weight inequality}
Let $(M,d)$ be a Hadamard space. Let $G:M\rightarrow (-\infty,\infty]$ be a convex lsc function. Let $\mathcal{R}$ be the space of geodesic rays in $M$ emanating from a fixed point $x_0\in M$. Define $\gga:\mathcal{R}\rightarrow (-\infty,\infty]$ by
\begin{equation}
\gga(\ell):=\lim_{t\to\infty}\frac{G(\ell_t)}{t}\,.
\end{equation}
As before, we may identify $\mathcal{R}$ for different $x_0$, the $\gga$ functionals for different $x_0$ correspond to each other.

For $\ell\in \mathcal{R}$, let 
\begin{equation}
\|\ell\|:=d(\ell_0,\ell_1)\,.
\end{equation}
This agrees with the definition in \eqref{eq:pnorm} for the Hadamard space $\mathcal{E}^2$.

We denote the trivial ray in $\mathcal{R}$ by $0$.
\begin{proposition}\label{prop:DonHis}
\begin{equation}\label{eq:DHtemp1}
\inf_{x\in M} |\partial G|(x)\geq \sup_{\ell\in \mathcal{R}\setminus\{0\}}\frac{-\mathbf{G}(\ell)}{\|\ell\|}\,.
\end{equation}
\end{proposition}
\begin{proof}
Take $\ell\in \mathcal{R}\setminus\{0\}$.
Fix $x_0\in M$. Then
\[
-\gga(\ell)\leq -\left.\ddt\right|_{t=0+}G(\ell_t)\leq |\partial G|(x_0)\|\ell\|\,,
\]
where the first inequality follows from the convexity of $G$, the second inequality follows from \eqref{eq:partialg}. Since $x_0$ is arbitrary, the inequality follows.
\end{proof}

This is also known as the moment-weight inequality in the general GIT setting.
\subsection{Weak Calabi flow}\label{subsec:wcf}
In this subsection, we explore the weak Calabi flow following \cite{BDL16}.

Fix a compact K\"ahler manifold $X$ and a K\"ahler form $\omega$ as before.

The following theorem is the basis of this part.
\begin{theorem}
The space $\mathcal{E}^2(X,\omega)$ is a Hadamard space.
\end{theorem}
This result is proved by Darvas in \cite{Dar17}.
See also \cite{Gue14} Theorem~3.11, Theorem~3.6.

The weak Calabi flow is an analogue of the Calabi flow recalled in the introduction. By definition, the weak Calabi flow is the weak gradient flow of the functional $M$ on $\mathcal{E}^2$. See \cite{BDL16} Section~6 for a thorough treatment.

We recall that for an initial value $\phi_0\in \mathcal{H}$, the weak Calabi flow coincides with the Calabi flow on the maximal existence time interval of the latter  (\cite{BDL16} Proposition~6.1).

Now we define a functional $\overline{Ca}:\mathcal{E}^2\rightarrow [0,\infty]$ as $|\partial M|$. As recalled above, $\overline{Ca}$ is lsc.

\begin{proposition}
For $\phi\in \mathcal{H}$, 
\[
\overline{Ca}(\phi)=Ca(\phi)\,.
\]
\end{proposition}
\begin{proof}
Recall that the evolution variation inequality also holds for the Calabi flow with smooth initial value (See \cite{He15} the equation below (2.4)). So \eqref{eq:ulb} also holds on the time interval where the Calabi flow is defined. Moreover, \eqref{eq:ulb} extends to $t=0$.

Now fix $\phi_t$ be a solution to the weak Calabi flow with $\phi_0\in \mathcal{H}$, since the flow coincides with the Calabi flow on a short time interval, we conclude that $M(\phi_t)$ is smooth in $t$ for small $t$, so by \eqref{eq:derg} and the fact that $\overline{Ca}$ is lsc,
\[
Ca(\phi_0)^2=-\dot{M}(\phi_0)\geq \overline{Ca}(\phi_0)^2\,.
\]
For the other inequality, by Proposition~\ref{prop:ulb},
\[
\overline{Ca}(\phi_0)\geq \frac{M(\phi_0)-M(\phi_t)}{d_2(\phi_t,\phi_0)}\geq Ca(\phi_t)\,.
\]
for $t>0$ small. Let $t\to 0+$, we conclude.
\end{proof}
From now on, we will no longer use the notation $\overline{Ca}$, we denote it simply as $Ca$.

Let $\phi_t$ be a solution to the weak Calabi flow with $M(\phi_0)<\infty$.
As we have recalled above, $Ca(\phi_t)$ is decreasing in $t$, so one can define
\begin{equation}\label{eq:Bcal}
B:=\lim_{t\to \infty}Ca(\phi_t)\,.
\end{equation}
According to Proposition~\ref{prop:Bind}, the value of $B$ is independent of the choice of $\phi_0$.

\subsection{Inverse Monge--Amp\`ere flow}
Now assume that we are in the Fano case, we recall the basic theory of the inverse Monge--Ampère flow following \cite{CHT17}.

The inverse Monge--Ampère flow is the gradient flow of $D$ on $\mathcal{H}$, namely,
\begin{equation}\label{eq:IMAF}
    \left\{
    \begin{aligned}
        \partial_t\varphi_t&=1-e^{\rho_{t}}\,,\\
        \left.\varphi_t\right|_{t=0}&=\varphi_0\,,
    \end{aligned}
    \right.
\end{equation}
where $\rho_t$ is short for $\rho_{\varphi_t}$. 
In the same spirit, we write $\omega_t=\omega_{\varphi_t}$.
We assume that $\varphi_0\in \mathcal{H}$.

\begin{theorem}[\cite{CHT17}]
The solution to \eqref{eq:IMAF} exists for $t\in [0,\infty)$ and is smooth.
\end{theorem}
One could of course define the weak gradient flow of $D$ as we did for $M$. But due to this theorem and a similar argument as \cite{BDL16} Proposition~6.1, the weak flow and the inverse Monge--Amp\`ere flow are exactly the same when the initial value lies in $\mathcal{H}$. As we will see, this is enough for our purpose.

Fix a smooth solution $\varphi_t$ to \eqref{eq:IMAF}. 
Note the following 
\[
-\ddt D(\varphi_t)=R(\varphi_t)^2\,.
\]
\begin{proposition}\label{prop:aprio}
\begin{enumerate}
    \item $E$ is constant along \eqref{eq:IMAF}.
    \item $R$ is decreasing along \eqref{eq:IMAF}.
    \item $M$ is decreasing along \eqref{eq:IMAF}.
\end{enumerate}
\end{proposition}
See \cite{CHT17} for a proof.

According to Proposition~\ref{prop:aprio}, $D$ is convex and decreasing along the flow.
Define
\begin{equation}\label{eq:B}
B:=\lim_{t\to \infty} R(\varphi_t)\in [0,\infty)\,.
\end{equation}
Again, $B$ is independent of the choice of $\varphi_0$.

\begin{remark}\label{rmk:Bpositive}
When $B>0$ ($B$ is defined in \eqref{eq:B}), $X$ does not admit K\"ahler--Einstein metrics. Otherwise, as is well-known, the K\"ahler--Einstein metric is a global minimizer of $D$, and as $D$ is convex and decreasing along $\varphi_t$, we infer that $B=0$, this is a contradiction. 

The same remark applies to the weak Calabi flow setting. Hence if $B>0$ ($B$ is defined in \eqref{eq:Bcal}), there is no cscK metric.
\end{remark}

\section{Proof of the main theorem}

\subsection{Analogue in finite dimensions}
Let us explain the idea of the proof in the finite dimensional setting. 

Let $G:\mathbb{R}^n\rightarrow \mathbb{R}$ be a smooth convex function. We may consider the gradient flow of $G$, namely
\[
\dot{x}_t=-\nabla G(x_t)\,.
\]
It is well-known that for any initial value $x_0\in \mathbb{R}^n$, there is always a smooth global solution.

Following the general theory of Hadamard spaces, we define the boundary $\mathbb{R}^n(\infty)$ as the set of equivalence classes of unit speed rays (in the usual sense) in $\mathbb{R}^n$, two rays are considered as equivalent if they are parallel in the sense that they are related by a translation. There is an obvious identification $\mathbb{R}^n(\infty)$ with the unit sphere $S^{n-1}$.

We can define a radial version of $G$, namely $\gga:\mathbb{R}^n(\infty)\rightarrow (-\infty,\infty]$ as follows: let $[\ell]\in \mathbb{R}^n(\infty)$,
take $x\in \mathbb{R}^n$, take a representative of $\ell$ of $[\ell]$ that emanates from $x$, define
\begin{equation}
\gga([\ell])=\lim_{t\to\infty}\frac{G(\ell_t)}{t}\,.
\end{equation}
It is easy to show that $\gga$ is independent of the choice of $x$. See the proof of \cite{DL18} Lemma~4.10.

Fix a solution to the flow, say $x_t$. Set $G(t)=G(x_t)$.

Then we claim that 
\begin{equation}\label{eq:toprove}
\left(-\lim_{t\to\infty} \dot{G}(t)\right)^{1/2}=\max\left\{0,\sup_{[\ell]\in \mathbb{R}^n(\infty)}
-\gga([\ell])\right\}\,.
\end{equation}
Let $\ell$ be a unit speed ray emanating from $x\in \mathbb{R}^n$. Then by Proposition~\ref{prop:DonHis}, we have
\[
-\gga([\ell])\leq=\left(-\dot{G}(0)\right)^{1/2}\,.
\]
Since $x$ is arbitrary, we conclude
\[
\left(-\lim_{t\to\infty} \dot{G}(t)\right)^{1/2}\geq \max\left\{0,\sup_{[\ell]\in \mathbb{R}^n(\infty)}
-\gga([\ell])\right\}\,.
\]

For the inverse direction, we may assume that
\begin{equation}\label{eq:Gdot}
\left(-\lim_{t\to\infty} \dot{G}(t)\right)^{1/2}>0\,.
\end{equation}
In this case, $|x_0-x_t|\to \infty$ as $t\to \infty$. Otherwise, let $y$ be a limit point of $x_t$, it is easy to see that $G(y)$ obtains the minimial value of $G$.  It is a general fact of the gradient flow that the left-hand side  of \eqref{eq:Gdot} is independent of the choice of $x_0$ (Proposition~\ref{prop:Bind}), so we find a contradiction by considering the flow starting at $y$.

By Proposition~\ref{prop:ulb},
we have the following control for $0\leq t<s$,
\[
\left(-\dot{G}(s)\right)^{1/2}\leq \frac{G(t)-G(s)}{|x_t-x_s|}
\leq 
\left(-\dot{G}(t)\right)^{1/2}\,.
\]

Now we claim that the sup on right-hand side  of \eqref{eq:toprove} is indeed obtained by a special direction $\ell^{\infty}$. The construction is as follows: connect $x_0$ and $x_s$ by a unit speed segment $\ell^{s}:[0,|x_0-x_s|]\rightarrow \mathbb{R}^n$. Fix $T>0$, it easy to see that the images of the maps $\ell^{s}|_{[0,T]}$  all lie in a fixed compact set when $s\geq T$, so we may take $s_i\to \infty$ so that the corresponding $\ell^{s_i}$ tends to another segment uniformly. Combining this with a Cantor diagonal argument, we arrive at a subsequence $s_i\to \infty$, so that the corresponding $\ell^{s_i}$ converge to a ray $\ell^{\infty}$ in uniformly on each compact time interval. We then calculate for $0<A<s$ that
\[
\left(-\lim_{t\to\infty} \dot{G}(t)\right)^{1/2}\leq  \left(- \dot{G}(s)\right)^{1/2}\leq \frac{G(0)-G(s)}{|x_0-x_s|}\leq 
\frac{G(0)-G(\ell^s_A)}{A}\,.
\]
Let $s\to \infty$ along the subsequence $s_i$ used to define $\ell^{\infty}$, we find
\[
\left(-\lim_{t\to\infty} \dot{G}(t)\right)^{1/2}\leq \frac{G(0)-G(\ell^{\infty}_A)}{A}\,.
\]
Let $A\to \infty$, we conclude
\[
\left(-\lim_{t\to\infty} \dot{G}(t)\right)^{1/2}\leq -\gga([\ell^{\infty}])\,.
\]
Hence equality in \eqref{eq:toprove} indeed holds.

It is not hard to generalize the proof to a general locally compact Hadamard space and to lsc and convex $G$. But in the situation we are interested in, the underlying space is $\mathcal{E}^2$, which is not locally compact. So one need some additional compactness theorem. In $\mathcal{E}^2$, the compactness is usually lacking, so we instead apply the compactness theorem for the level set of $H$ in $\mathcal{E}^1$ proved in \cite{BBEGZ16}. The details will be treated in the subsequent subsections.

\subsection{An abstract version}\label{subsec:abs}
Let $(M,d)$ be a Hadamard space.
Let $\sigma$ be a topology on $M$. We say $\sigma$ is \emph{compatible} with $(M,d)$ if the followings hold:
\begin{enumerate}
    \item $\sigma$ is a Hausdorff topology.
    \item $\sigma$ is weaker that the $d$-topology. Moreover, let $x_j$ be a bounded sequence in $(M,d)$, such that $x_j\to x\in M$ with respect to the $\sigma$-topology. Then $x_j\to x$ with respect to the weak topology.
    \item For any bounded $\sigma$-converging sequences $x_j\to x$, $y_j\to y$ in $M$,
    \[
        d(x,y)\leq \varliminf_{j\to\infty} d(x_j,y_j)\,.
    \]
    \item Let $(x^j_t)_{t\in [0,1]}$ be geodesics in $M$ for any $j\geq 1$. Assume that there are $x_0,x_1\in M$, such that $x_0^j\to x_0$, $x_1^j\to x_1$ in $\sigma$-topology. Let $(x_t)_{t\in [0,1]}$ be the geodesic from $x_0$ to $x_1$. Then for any $t\in [0,1]$, $x^j_t\to x_t$ in $\sigma$-topology.
\end{enumerate}

\begin{theorem}\label{thm:abst}
Let $(M,d)$ be a Hadamard space. Let $\sigma$ be a topology on $M$ compatible with $(M,d)$.
Let $F,G:M\rightarrow(-\infty,\infty]$ be two convex lsc functions such that $F\leq G$ and such that $G$ is decreasing along the gradient flow of $F$. Fix an arbitrary point $x_0\in \Dom G$.
Assume that for any constant $C>0$, the following set
\[
\mathcal{K}_{C}:=\left\{\,x\in M: d(x,x_0)\leq C, G(x)\leq C \,\right\}\subseteq M\,.
\]
is $\sigma$-sequentially compact. 
Then
\begin{equation}\label{eq:tbp}
\inf_{x\in M} |\partial F|(x)=\max\left(0,\max_{\ell\in \mathcal{R}\setminus\{0\}}\frac{-\mathbf{F}(\ell)}{\|\ell\|}\right)\,.
\end{equation}
\end{theorem}
Here $\mathcal{R}$ denotes the space of all geodesic rays emanating from $x_0$ and $0$ denotes the trivial ray in $\mathcal{R}$.
The functional $\mathbf{F}:\mathcal{R}\rightarrow (-\infty,\infty]$ is defined by
\begin{equation}\label{eq:genF}
\mathbf{F}(\ell):=\lim_{t\to\infty}\frac{F(\ell_t)}{t}\,.
\end{equation}
The norm of a geodesic ray $\ell$ is defined as
\[
\|\ell\|:=d(\ell_0,\ell_1)\,.
\]
As before, we identify $\mathcal{R}$ with respect to different $x_0$.
The functional $\mathbf{F}$ does not depend on the choice of $x_0$.

\begin{proof}
Let $(x_t)_{t\geq 0}$ be the gradient flow of $F$ with starting point $x_0$. 

\textbf{Case 1}. Assume that $d(x_0,x_t)$ is bounded.

In this case, by our assumption, the set $\{x_t:t\in [0,\infty)\}$ is weakly relatively compact. In particular, we can take $t_j\to \infty$ ($j\geq 1$), such that $x_{t_j}$ converges weakly to $x_{\infty}\in M$ as $j\to\infty$. By \cite{Bac13} Lemma~3.1, $F$ is weakly lsc, so 
\[
F(x_{\infty})\leq \varliminf_{j\to\infty}F(x_{\infty})\,.
\]
By \cite{Bac14} Proposition~5.1.12, we conclude that $x_{\infty}$ is indeed a minimizer of $F$. 
Also observe that by the same argument, $G(x_{\infty})<\infty$.
In particular, we can replace $x_0$ by $x_{\infty}$. In this case, both sides of \eqref{eq:tbp} are $0$.

\textbf{Case 2}. Assume that $d(x_0,x_t)$ is not bounded. Then we can take $t_i\to \infty$ ($i\geq 1$) so that $d(x_0,x_{t_i})\to \infty$. Replacing $x_0$ with $x_{\epsilon}$ for a small $\epsilon>0$, we may assume that Proposition~\ref{prop:ulb} holds up to $t=0$.

For each $t\geq 0$, let $(\ell^t_s)_{s\in [0,d(x_0,x_t)]}$ be the unit-speed geodesic segment from $x_0$ to $x_t$. By the convexity of $G$, we get
\[
G(\ell^t_s)\leq \frac{d(x_0,x_t)-s}{d(x_0,x_t)}G(x_0)+\frac{s}{d(x_0,x_t)}G(x_t)\,.
\]
By our assumption, $G(x_t)\leq G(x_0)$. So
\[
G(\ell^t_s)\leq G(x_0)<\infty\,.
\]
For a fixed $s_0$, we can take large enough $i$ so that $d(x_0,x_{t_j})>s_0$ for any $j\geq i$. Then there is a constant $C>0$ so that $\ell^{t_j}_s\in \mathcal{K}_C$ for any $j\geq i$, $s\in [0,s_0]$. By the compactness assumption, the Ascoli--Arzelà theorem (\cite{AGS08} Proposition~3.3.1) and the diagonal argument, after possibly replacing $t_j$ by a subsequence, we may assume that there is a geodesic ray $\ell^{\infty}\in \mathcal{R}$, such that
$\ell^{t_j}_s$ $\sigma$-converges to $\ell^{\infty}_s$ as $j\to\infty$ for all $s\geq 0$.

Fix $s\geq 0$, when $t_j\geq s$,
\[
\inf_{x\in M}|\partial F|(x)\leq |\partial F|(x_{t_j})\leq \frac{F(x_0)-F(x_{t_j})}{d(x_0,x_{t_j})}\leq \frac{F(x_0)-F(\ell^{t_j}_s)}{s}\,,
\]
where the second inequality follows from Proposition~\ref{prop:ulb}, the third follows from the convexity of $F$. 
Let $j\to\infty$, since $F$ is weakly lsc, we get
\[
\inf_{x\in M}|\partial F|(x)\leq  \frac{F(\ell^{\infty}_0)-F(\ell^{\infty}_s)}{s}\,,
\]

Let $s\to\infty$, we conclude that
\[
\inf_{x\in M}|\partial F|(x)\leq -\mathbf{F}(\ell^{\infty})\,.
\]
When $\ell^{\infty}$ is trivial, we conclude immediately. Now assume that $\ell^{\infty}$ is not trivial. By Proposition~\ref{prop:weakineq}, $\|\ell^{\infty}\|\leq 1$. So
\[
\inf_{x\in M}|\partial F|(x)\leq -\mathbf{F}(\ell^{\infty})\leq \frac{-\mathbf{F}(\ell^{\infty})}{\|\ell^{\infty}\|}\leq \inf_{x\in M}|\partial F|(x)\,,
\]
where the last inequality follows from Proposition~\ref{prop:DonHis}. Now \eqref{eq:tbp} follows.
\end{proof}
As a by-product of the proof, we find that if 
\[
\inf_{x\in M} |\partial F|(x)>0\,,
\]
then
\begin{equation}\label{eq:e2}
\|\ell^{\infty}\|=1\,.
\end{equation}

We call the geodesic rays that minimizes $\mathbf{F}(\ell)/\|\ell\|$ the \emph{Darvas--He geodesic rays}.

\begin{corollary}\label{corollary:stronger}
Assume that
\begin{equation}\label{eq:uns1}
\max_{\ell\in \mathcal{R}\setminus\{0\}}\frac{-\mathbf{F}(\ell)}{\|\ell\|}>0
\end{equation}
and that the maximizer is unique. Then for any $s\geq 0$, $\ell_s^t$ constructed in the previous proof starting from $x_{\epsilon}$ for any $\epsilon>0$ converges to $\ell^{\infty}_s$ in $M$ as $t\to \infty$, where $\ell^{\infty}$ is moved parallelly so that $\ell^{\infty}_0=x_{\epsilon}$.
\end{corollary}
\begin{proof}
We use the same notations as in the proof of Theorem~\ref{thm:abst}. By \eqref{eq:uns1}, we are in Case 2. By replacing $x_0$ by $x_{\epsilon}$, we may set $\epsilon=0$.

By \cite{Bac14} Proposition~3.1.6, Theorem~\ref{thm:abst} and \eqref{eq:e2}, it suffices to prove that for any $s\geq 0$, $\ell^{t}_s$ converges weakly to $\ell^{\infty}_s$ as $t\to \infty$. 
For this purpose, it suffices to prove that for any sequence $t_i\to\infty$, we can find a subsequence $t_{n_i}\to \infty$ such that
$\ell^{t_{n_i}}_s$ converges weakly to $\ell^{\infty}_{s}$.

Due to \eqref{eq:uns1}, we have
\[
\lim_{i\to \infty}d(x_0,x_{t_i})=\infty\,.
\]
So we can construct a Darvas--He geodesic $\ell$ from a subsequence $t_{n_i}$. We know that $\ell^{t_{n_i}}_{s}$ converges weakly  to $\ell_{s}\in M$.
By the uniqueness of the maximizer, we conclude that $\ell_{s}=\ell^{\infty}_s$. The result follows.
\end{proof}

\subsection{Proof of Theorem~\ref{thm:main}}
Now to get Theorem~\ref{thm:main}, one takes $(M,d)$ to be $(\mathcal{E}^2,d_2)$, $G=M$ and $F$ is $M$ for the weak Calabi flow, $D$ for the inverse Monge--Amp\`ere flow. 
It remains to check the compactness properties of $\mathcal{K}_C$.

\begin{lemma}\label{lma:aux}
Let $\varphi_j$ ($j\in \mathbb{N}$) be a bounded sequence in $\mathcal{E}^2$. Let $\varphi\in \mathcal{E}^1$. Assume that $\varphi_j\to \varphi$ in $\mathcal{E}^1$. Then $\varphi\in \mathcal{E}^2$. Moreover, for any $\psi\in \mathcal{E}^2$, 
\[
d_2(\psi,\varphi)\leq \varliminf_{j\to\infty}d_2(\psi,\varphi_j)\,.
\]
\end{lemma}
\begin{proof}
Since $\varphi_j \to \varphi$ in $\mathcal{E}^1$, we know that 
\begin{equation}\label{eq:t}
\varphi_j\to \varphi \,\, a.e.\,, \quad \left|\sup_X \varphi_j\right|\leq C\,.
\end{equation}
Define
\[
w_j=\operatorname*{sup*}_{\!\!\!\!i\geq j} \varphi_i\,.
\]
Then \eqref{eq:t} together with the Choquet lemma implies that $w_j$ decreases and converges to $\varphi$ a.e.. 

According to \cite{Dar15} Lemma~4.16, in order to prove that $\varphi\in \mathcal{E}^2$, it suffices to prove that $d_2(0,w_j)$ is bounded. According to \eqref{eq:ulb}, this is equivalent to prove
\[
\int_X |w_j|^2 \,\omega^n \leq C\,, \quad \int_X |w_j|^2 \,\omega^n_{w_j} \leq C\,.
\]
For the former, it suffices to consider the negative part of $w_t$, which is bounded from below by $\varphi_j$, so it suffices to prove
\[
\int_X |\varphi_j|^2 \,\omega^n\leq C\,. 
\]
This follows again from \eqref{eq:ulb} and the assumption that $\varphi_j$ is bounded in $\mathcal{E}^2$.

For the latter, according to \cite{GZ07} and \eqref{eq:ulb}, we have
\[
\int_X |w_j|^2 \,\omega^n_{w_j}\leq C\int_X |\varphi_j|^2 \,\omega^n_{\varphi_j}+C\leq C\,.
\]
So we conclude that $\varphi\in \mathcal{E}^2$.

According to \cite{BDL17} Theorem~5.3. $\varphi$ is the weak limit of $\varphi_j$. So we conclude by Proposition~\ref{prop:weakineq}.
\end{proof}

Recall the following version of the compactness theorem of \cite{BBEGZ16}.
\begin{theorem}\label{thm:BBEGZc}
For any $C>0$, $\varphi_0\in \mathcal{E}^1$, the set
\[
K_C:=\{\varphi\in \mathcal{E}^1: M(\varphi)\leq C, d_1(\varphi,\varphi_0)\leq C\}\subseteq \mathcal{E}^1
\]
is compact with respect to the strong topology.
\end{theorem}
\begin{proof}
Let $\varphi\in \mathcal{E}^1$ be a potential such that $M(\varphi)\leq C$, $d_1(\varphi,\varphi_0)\leq C$.

By \cite{DH17} Proposition~2.5\footnote{It was only stated for $\varphi\in \mathcal{H}$, but since $E_R$ is continuous on $\mathcal{E}^1$, it also holds for $\varphi\in \mathcal{E}^1$.}, $H(\varphi)\leq C$ for a constant $C_1$. Moreover, according to \cite{DDNL18b} Lemma~3.9, 
\[
\left|\sup \varphi\right|\leq C_2\,.
\]
So according to \cite{BBEGZ16} Theorem~2.17, Proposition~2.6, for any sequence $\varphi_j\in K_C$, up to selecting a subsequence, we may assume that $\varphi_j$ converges to $\varphi\in \mathcal{E}^1$ in the strong topology. Now as $M$ is lsc, we conclude that
\[
M(\varphi)\leq C\,,
\]
so $\varphi\in K_C$. This concludes the proof.
\end{proof}
\begin{corollary}\label{cor:BBEGZc}
For any $C>0$, $\varphi_0\in \mathcal{E}^2$, the set
\[
\mathcal{K}_C:=\left\{\varphi\in \mathcal{E}^2: M(\varphi)\leq C\,,\,\, d_2(\varphi,\varphi_0)\leq C\right\}\subseteq \mathcal{E}^2
\]
is compact with respect to $d_1$-topology.
\end{corollary}
\begin{proof}
Let $\varphi_j\in \mathcal{K}_C$. By Theorem~\ref{thm:BBEGZc}, up to selecting a subsequence, we may assume that $\varphi_j$ converges to $\varphi\in \mathcal{E}^1$ in the $d_1$-topology. Moreover, $M(\varphi)\leq C$.
Then according to Lemma~\ref{lma:aux}, we have $\varphi\in \mathcal{E}^2$ and $d_2(\varphi,\varphi_0)\leq C$.
\end{proof}
\begin{proposition}\label{prop:compd1}
The $d_1$-topology on $\mathcal{E}^2$ is compatible with $(\mathcal{E}^2,d_2)$.
\end{proposition}
For the definition of compatibility, see Section~\ref{subsec:abs}.

\begin{proof}
Condition~(1) is obvious. For Condition~(2),  recall that for a bounded sequence in $\mathcal{E}^2$, convergence in $\mathcal{E}^1$ implies convergence in the weak topology (\cite{BDL17} Theorem~1.6).
Condition~(3) follows from \cite{Bac13} Lemma~3.1 and Condition~(2). Finally, Condition~(4) follows from \cite{BBJ15} Proposition~1.11.
\end{proof}

\begin{proof}[Proof of Theorem~\ref{thm:main}]
Let $(M,d)=(\mathcal{E}^2,d_2)$.
Let $\sigma$ be the $d_1$-topology on $\mathcal{E}^1$. By Proposition~\ref{prop:compd1}, $\sigma$ is compatible with $(M,d)$.

(1) We apply Theorem~\ref{thm:abst} with $F=G=M$. The compactness condition is guaranteed by Corollary~\ref{cor:BBEGZc}.

(2) Recall that $M$ is decreasing along the inverse Monge--Amp\`ere flow according to \cite{CHT17} Lemma~4.6. 
We apply Theorem~\ref{thm:abst} with $F=D$, $G=M$. 
The compactness condition is guaranteed by Corollary~\ref{cor:BBEGZc}. As the inverse Monge--Amp\`ere flow admits global smooth solutions, by Proposition~\ref{prop:Bind}, we have
\[
\inf_{\varphi\in \mathcal{H}}R(\varphi)=\inf_{\varphi\in \mathcal{E}^2}R(\varphi)\,.
\]

Finally observe that in the Fano case, 
\[
\max_{\ell\in \mathcal{R}^{2}\setminus\{0\}} \frac{-\mathbf{D}(\ell)}{\|\ell\|}=0
\]
implies that $X$ is K-semistable (See \cite{Ber16}).
\end{proof}
\begin{remark}
In contrast to general Hadamard spaces, in $\mathcal{E}^2$ we have geodesic rays of the form $(Ct)_{t\geq 0}$. These rays have vanishing $\mm$. So 
\[
\max_{\ell\in \mathcal{R}^2\setminus\{0\}}\frac{-\mm(\ell)}{\|\ell\|}
\]
is always non-negative. Similar remark holds for $\dd$.
\end{remark}

\begin{remark}\label{rmk:infsup}
If the Calabi flow admits a global smooth solution, it will follow from the same proof that
\[
\inf_{\phi\in\mathcal{H}}Ca(\phi)=\max_{\ell\in \mathcal{R}^2\setminus\{0\}}\frac{-\mm(\ell)}{\|\ell\|}\,.
\]
\end{remark}

\subsection{Uniqueness of the maximizer}\label{subsec:pfc}
\begin{definition}\label{def:geoduns}
We say $(X,\omega)$ is \emph{geodesically unstable} if 
\[
\max_{\ell\in \mathcal{R}^{2}\setminus\{0\}} \frac{-\mm(\ell)}{\|\ell\|}>0\,.
\]
Otherwise, we say $(X,L)$ is \emph{geodesically semistable}.
\end{definition}
According to \cite{DL18} Theorem~1.5, $(X,\omega)$ is geodesically unstable if{f} there is a $C^{1,\bar{1}}$ geodesic ray $\ell$, such that $\mm(\ell)<0$.

\begin{theorem}\label{thm:Had}
$\mathcal{R}^2$ is a Hadamard space.
\end{theorem}
\begin{proof}
It is known that $\mathcal{R}^2$ is a complete geodesic metric space (\cite{DL18} Theorem~1.3, Theorem~1.4). So it suffices to prove that $\mathcal{R}^2$ satisfies the CAT(0)-inequality. More concretely, we need to show: if $\ell,\ell^s\in \mathcal{R}^2$ ($s\in [0,1]$), $\ell^s$ is a geodesic segment in $\mathcal{R}^2$, then for any $s\in [0,1]$, we have
\[
d_2^c(\ell,\ell^s)^2\leq (1-s)d_2^c(\ell,\ell^0)^2+sd_2^c(\ell,\ell^1)^2-s(1-s)d_2^c(\ell^0,\ell^1)^2\,.
\]
Without loss of generality, we may assume that the starting point of geodesic rays in $\mathcal{R}^2$ are $0$. We recall the construction of $\ell^s$ from $\ell^0$ and $\ell^1$. For each $t\geq 0$, let $(\ell^{,t}_s)_{s\in [0,1]}$ be the geodesic segment from $\ell^0_t$ to $\ell^1_t$. Let $(L^{t,s}_T)_{T\in [0,t]}$ be the geodesic segment from $0$ to $\ell^{,t}_s$. Then for any fixed $T\geq 0$, $L^{t,s}_T$ for $t\to \infty$ has a unique limit, the limit is defined to be $\ell^{s}_T$. 

Now for any $T\geq 0$,
\[
\frac{1}{T^2}d_2(\ell_T,\ell_T^s)^2=\lim_{t\to\infty}\frac{1}{T^2}d_2(\ell_T,L_T^{t,s})^2\leq \varlimsup_{t\to\infty} \frac{1}{t^2}d_2(\ell_t,\ell_s^{,t})\,,
\]
where the last inequality follows from \cite{DL18} (1).

Now since $\mathcal{E}^2$ is a Hadamard space, we find for any $t\geq 0$
\[
d_2(\ell_t,\ell_s^{,t})^2\leq  (1-s)d_2(\ell_t,\ell^0_t)^2+s d_2(\ell_t,\ell^1_t)^2-s(1-s)d_2(\ell^0_t,\ell^1_t)\,.
\]
Hence
\[
\begin{split}
\frac{1}{T^2}d_2(\ell_T,\ell_T^s)^2\leq \varlimsup_{t\to\infty} \frac{1}{t^2}\left((1-s)d_2(\ell_t,\ell^0_t)^2+s d_2(\ell_t,\ell^1_t)^2-s(1-s)d_2(\ell^0_t,\ell^1_t)\right)\\
=(1-s)d_2^c(\ell,\ell^0)^2+sd_2^c(\ell,\ell^1)^2-s(1-s)d_2^c(\ell^0,\ell^1)^2\,.
\end{split}
\]
Let $T\to\infty$, we conclude.
\end{proof}

\begin{proof}[Proof of Corollary~\ref{cor:main}]
We only prove part 1, since part 2 is similar.

Assume that $(X,\omega)$ is geodesically unstable. Let $\varphi\in \mathcal{E}^2$ with $M(\varphi)<\infty$.
Let $\ell^0,\ell^1$ be two \emph{different} minimizers of $\mathbf{M}$ on the unit sphere. 
Let $(\ell^{s})_{s\in [0,1]}$ be the unique $d_2^c$-geodesic between them. Since $\mm$ is convex in $\mathcal{R}^2$ (\cite{DL18} Theorem 4.11), we have
\begin{equation}
    -\mm(\ell^{s})\geq \inf_{\phi\in \mathcal{E}^2}Ca(\phi).
\end{equation}
By the CAT(0)-inequality of $\mathcal{R}^2$,
\[
\|\ell^{s}\|<1\,,\quad s\in (0,1)\,.
\]
Hence
\[
\frac{-\mm(\ell^{s})}{\|\ell^{s}\|}> \inf_{\phi\in \mathcal{E}^2}Ca(\phi)\,.
\]
This is a contradiction.
\end{proof}

In particular, the conditions of Corollary~\ref{corollary:stronger} are satisfied.

\section{Further remarks and conjectures}\label{sec:rmk}

\subsection{Relations between Theorem~\ref{thm:main} and Donaldson's conjecture}\label{subsec:TC}
In this section, we assume that the polarization of $X$ is integral, namely, coming from an ample line bundle $L$ on $X$. This assumption is not essential, but makes notations simpler.

Let $\mathcal{H}^{\NA}$ be the space of non-Archimedean metrics defined in \cite{BHJ19}, \cite{BHJ17}. Recall that there is a natural map $\iota:\mathcal{H}^{\NA}\rightarrow \mathcal{R}^p$ for $p\geq 1$. Moreover, the geodesic rays in the image of $\iota$ have $C^{1,1}$-regularity (\cite{CTW18}).
This construction dates back to \cite{PS07}. See also \cite{RWN14}, \cite{DDNL18}. 

The map admits a natural extension to an embedding
$\iota:\mathcal{E}^{1,\NA}\rightarrow \mathcal{R}^1$.
See Theorem~6.6 in \cite{BBJ15}. Here $\mathcal{E}^{1,\NA}$ is the non-Archimedean analogue of the usual $\mathcal{E}^1$ space. For the precise definition, we refer to \cite{BBJ15}, \cite{BJ18}, \cite{Bou18} and references therein. 

Now let us explain the relation between Donaldson's conjecture (i.e. equality in \eqref{eq:Don}, \eqref{eq:Don1}) and Theorem~\ref{thm:main}.

Let $\ell$ be the image of a non-Archimedean metric $\psi\in \mathcal{H}^{\NA}$ under the map $\iota$. According to Theorem~1.2 in \cite{His16}, 
\[
\|\psi\|_{L^2}^2=\frac{1}{V}\int_X |\dot{\ell}_0|^2 \,\omega_{\ell_0}^n\,.
\]
Since we already know that $\ell$ has $C^{1,1}$ regularity, it follows from \cite{Dar15} Lemma~4.11 that 
\[
\frac{1}{V}\int_X |\dot{\ell}_0|^2\, \omega_{\ell_0}^n=\|\ell\|^2\,.
\]

According to \cite{BHJ17} Proposition~2.8,
\[
\mathrm{DF}(\mathcal{X},\mathcal{L})=M^{\NA}(\psi)\,,
\]
where $(\mathcal{X},\mathcal{L})$ is a normal representative of $\psi$ with reduced central fibre. This shows the equivalence between \eqref{eq:Don1} and \eqref{eq:Don}. 

\begin{proposition}\label{prop:mmrel}
Notations as above, then
\[
\mm(\ell)\leq M^{\NA}(\psi)\,.
\]
\end{proposition}
\begin{proof}
According to \cite{BDL16} (4.2) and (4.3), we have a subgeodesic ray $\tilde{\ell}_t$, so that
\[
M(\tilde{\ell}_{t})=\mathrm{DF}(\mathcal{X},\mathcal{L})t+\mathcal{O}(1)\,,\quad d_2(\ell_{t},\tilde{\ell}_t)\leq C\,.
\]
For each $t>0$, let $[0,t]\ni a\mapsto v^t_a$ be the $d_2$-geodesic connecting $\ell_0$ to $\tilde{\ell}_t$. Let $\ell'$ be the geodesic ray with $\ell'_0=\tilde{\ell}_0$, which is parallel to $\ell$. The existence and uniqueness of $\ell'$ is guaranteed by Proposition~4.1 in \cite{DL18}. Let $[0,t]\ni a\mapsto u^{t}_a$ be the $d_2$-geodesic connecting $\ell_0$ to $\ell'_t$. As in the proof of \cite{DL18} Proposition~4.1, for fixed $a\geq 0$, $u_a^t\to \ell_a$ as $t\to \infty$. Now by \cite{DL18} (1), 
\[
d_2(u^t_a,v^t_a)=\mathcal{O}(1/t)\,.
\]
Hence we conclude
\[
v_a^t\to \ell_a\,,\quad t\to\infty\,.
\]
By the convexity of $M$, we find
\[
M(v_a^t)\leq \left(1-\frac{a}{t} \right)M(\ell_0)+\frac{a}{t}M(\tilde{\ell}_t)\,.
\]
Let $t\to\infty$ and use the fact that $M$ is lsc, we find
\[
\frac{M(\ell_a)}{a}\leq \frac{M(\ell_0)}{a}+\mathrm{DF}(\mathcal{X},\mathcal{L})\,.
\]
Finally, let $a\to \infty$, we conclude
\[
\mm(\ell)\leq \mathrm{DF}(\mathcal{X},\mathcal{L})\,.
\]
\end{proof}
\begin{remark}
The reverse inequality is recently proved by Chi Li in \cite{Li20}.
\end{remark}

\begin{conjecture}\label{conj:a}\footnote{The conjecture is true by the recent work \cite{Li20}.}
The Darvas--He geodesic lies in $\iota(\mathcal{E}^{1,\NA})$.
\end{conjecture}
In terms of the terminology of \cite{BBJ15}, we conjecture that the Darvas--He geodesic is \emph{maximal}.

Observe that Donaldson's conjecture (equality in \eqref{eq:Don1} and \eqref{eq:Don}) will follow from our result if the followings are true:
\begin{enumerate}
\item Conjecture~\ref{conj:a} is true and we have the following recovery property: for each $\ell\in \mathcal{R}^2\cap \iota(\mathcal{E}^{1,\NA})$, one could find a sequence $\ell^j$ in $\iota(\mathcal{H}^{\NA})$ such that $d_2^c(\ell,\ell^j)\to 0$, and such that
\[
\mm(\ell^j)\to \mm(\ell)\,.
\]
\item
Chen's conjecture is true: the Calabi flow admits long time smooth solution for an arbitrary smooth initial value (See Remark~\ref{rmk:infsup}).

\end{enumerate}

A positive result in this direction is recently proved by Darvas and Lu (\cite{DL18} Theorem~1.5). They showed that $\mathcal{R}^{1,\bar{1}}$ (the space of $C^{1,\bar{1}}$ geodesics) is dense in $\mathcal{R}^p$ for any $p\in [1,\infty)$. Moreover, a recovery property holds in this case.

Due to Theorem~\ref{thm:Had}, one can study the gradient flow of $\mathbf{M}$ on $\mathcal{R}^2$. This flow can be properly called the \emph{radial Calabi flow}. The behaviour of this flow will be closely related to our conjecture. 

\subsection{Harnack estimate}
We restrict our discussion to the inverse Monge--Amp\`ere flow here.

It is natural to guess that the Darvas--He geodesic rays that we construct should be locally bounded. By using Theorem~3.4 in \cite{Da17}, this will follow from a lower bound 
\[
\inf_X \varphi_t \geq -C t-C
\]
for a solution $\varphi_t$ to \eqref{eq:IMAF}.

The proof of \emph{a priori} bound of $\inf_X \varphi_t$ on finite time intervals in \cite{CHT17} is by means of contradiction, and it seems impossible to get qualitative bounds using their methods. 

A similar situation exists for K\"ahler--Ricci flows. However, in that case, the Sobolev constant along the flow is uniformly bounded, as a consequence of the monotonicity of the Perelman's W-entropy (See~\cite{Ye07} for details).
Then applying the usual Moser iteration, we arrive at a Harnack inequality (See~\cite{Rub09}, for example).

The problem for the inverse Monge--Ampère flow is that, the Perelman entropy, in its original form, is not monotone. And there does not seem to be any method to control the Sobolev constant in this case.  

We also notice that it is easy to deduce a lower bound exponential in $t$ using the Moser--Trudinger inequality \cite{BB11} and Ko{\l}odziej's $L^{\infty}$-estimate. See \cite{BEGZ10} for an explicit version of Ko{\l}odziej's estimate.

If the Harnack estimate does hold, we conclude immediately that the Darvas--He geodesic $\ell^{(t)}$ is non-trivial. So we get plenty of criteria for the existence of K\"ahler--Einstein metrics. 

Similar remarks hold also in the weak Calabi flow setting. Note that we do not require that the Calabi flow has a global smooth solution.

\printbibliography

\bigskip
  \footnotesize

  Mingchen Xia, \textsc{Department of Mathematics, Chalmers Tekniska Högskola, G\"oteborg}\par\nopagebreak
  \textit{E-mail address}, \texttt{xiam@chalmers.se}\par\nopagebreak
  \textit{Homepage}, \url{http://www.math.chalmers.se/~xiam/}.

\end{document}